\documentclass[a4paper]{amsart}

\usepackage[utf8]{inputenc}
\usepackage[english]{babel}

\usepackage[T1]{fontenc}
\usepackage{lmodern}  

\usepackage{my-math-symbol}
\usepackage{my-cleveref}
\usepackage{my-equation-numbering}
\usepackage{comment}
\usepackage{tikz}
\usetikzlibrary{cd}
\usepackage[colorinlistoftodos]{todonotes}

\usepackage[non-sorted-cites,initials,alphabetic,nobysame]{amsrefs}

\AtBeginDocument{%
	\def\MR#1{}
}

\title{Chern classes of spherical CR manifolds}
\author{Yuya Takeuchi}
\address{Division of Mathematics \\ Faculty of Pure and Applied Sciences \\ University of Tsukuba
	\\ 1-1-1 Tennodai, Tsukuba, Ibaraki 305-8571 Japan}
\email{ytakeuchi@math.tsukuba.ac.jp, yuya.takeuchi.math@gmail.com}

\subjclass[2010]{32V05; 57R17}

\keywords{spherical CR structure, Chern class}

\thanks{This work was supported by JSPS Research Fellowship for Young Scientists
and JSPS KAKENHI Grant Number JP19J00063 and JP21K13792.}

\begin{document}

\begin{abstract}
	We first construct closed spherical CR manifolds of dimension at least five
	having non-trivial first Chern class with real coefficients.
	We next prove a constraint on Chern classes with real coefficients
	of (not necessarily closed) spherical CR manifolds.
	Finally,
	we obtain a topological obstruction to the existence of spherical CR structures
	on co-oriented contact manifolds.
\end{abstract}

\maketitle

\section{Introduction}
\label{section:introduction}

Locally conformally flat manifolds,
or $(PO(n+1, 1), S^{n})$-manifolds,
play a crucial role in conformal geometry.
Chern and Simons~\cite{Chern-Simons1971} and Avez~\cite{Avez1970}
have proved that the Pontrjagin classes with real coefficients
are trivial on any locally conformally flat manifold.
In this paper,
we consider a similar problem for ``locally CR flat'' manifolds,
which are known as \emph{spherical CR manifolds},
or $(PU(n+1, 1), S^{2n+1})$-manifolds.

We will first construct closed spherical CR manifolds of dimension at least five
having non-trivial first Chern class with real coefficients.
In particular,
vanishing of characteristic classes does not occur for spherical CR manifolds in general
as opposed to locally conformally flat manifolds.

\begin{theorem}
\label{thm:spherical-CR-with-non-trivial-first-Chern-class}
	For each integer $n \geq 2$,
	there exists a closed spherical CR manifold $(S, T^{1,0}S)$ of dimension $2n+1$
	with $c_{1}(T^{1,0}S) \neq 0$ in $H^{2}(S, \bbR)$.
	In particular,
	$(S, T^{1, 0} S)$ has no pseudo-Einstein contact forms.
\end{theorem}

Note that there exist no known examples of spherical CR three-manifolds
having non-trivial first Chern class with real coefficients.

However,
we will obtain a constraint on higher Chern classes of spherical CR manifolds
in terms of the first Chern class.

\begin{theorem}
\label{thm:Chern-classes-of-spherical-CR-structure}
	Let $(M, T^{1,0}M)$ be a spherical CR manifold of dimension $2n+1$.
	Then
	\begin{equation}
		c_{k}(T^{1,0}M)
		= \frac{1}{(n+2)^{k}} \binom{n+2}{k} c_{1}(T^{1,0}M)^{k}
	\end{equation}
	in $H^{2k}(M, \bbR)$.
	In particular,
	\begin{equation}
		c_{2}(T^{1,0}M)
		= \frac{n+1}{2(n+2)} c_{1}(T^{1,0}M)^{2}
	\end{equation}
	in $H^{4}(M, \bbR)$.
\end{theorem}

The following proposition implies that
it is essential to consider the cohomology with real coefficients.

\begin{proposition}
\label{prop:counterexample-to-integer-coefficients}
	For each integer $n \geq 2$,
	there exists a closed spherical CR manifold $(S, T^{1,0}S)$ of dimension $2n+1$
	such that
	\begin{equation}
		2(n+2) c_{2}(T^{1,0}S)
		\neq (n+1) c_{1}(T^{1,0}S)^{2}
	\end{equation}
	in $H^{4}(S, \bbZ)$.
\end{proposition}

A spherical CR structure $T^{1,0}M$ has the underlying co-oriented contact structure $\xi \coloneqq \Re T^{1,0}M$.
Conversely,
it is natural to ask whether a given co-oriented contact structure $\xi$ admits a spherical CR structure.
\cref{thm:Chern-classes-of-spherical-CR-structure} gives a topological obstruction to this problem.

\begin{theorem}
\label{thm:topological-obstruction-to-spherical-CR-structure}
	Let $(M, \xi)$ be a $(2n+1)$-dimensional co-oriented contact manifold.
	If $\xi$ admits a spherical CR structure,
	then
	\begin{equation}
		c_{k}(\xi)
		= \frac{1}{(n+2)^{k}} \binom{n+2}{k} c_{1}(\xi)^{k}
	\end{equation}
	in $H^{2k}(M, \bbR)$.
	In particular,
	\begin{equation}
	\label{eq:second-Chern-contact-with-spherical}
		c_{2}(\xi)
		= \frac{n+1}{2(n+2)} c_{1}(\xi)^{2}
	\end{equation}
	in $H^{4}(M, \bbR)$.
\end{theorem}

This theorem follows from \cref{thm:Chern-classes-of-spherical-CR-structure}
and the fact that $c(\xi) = c(T^{1,0}M)$.
We apply this for a Stein fillable contact manifold.

\begin{proposition}
\label{prop:Stein-fillable-with-no-spherical-CR}
	For each integer $n \geq 3$,
	there exists a Stein fillable contact manifold $(M, \xi)$ of dimension $2n+1$ such that
	\begin{equation}
		c_{2}(\xi)
		\neq \frac{n+1}{2(n+2)} c_{1}(\xi)^{2}
	\end{equation}
	in $H^{4}(M, \bbR)$.
	In particular,
	$\xi$ admits no spherical CR structures.
\end{proposition}

Note that there exist no Stein fillable contact manifolds of dimension five violating \cref{eq:second-Chern-contact-with-spherical}.
This is because both the second Chern class and the square of the first Chern class
with real coefficients must vanish on any five-dimensional Stein fillable (more generally, holomorphically fillable) contact manifold;
see~\cite{Takeuchi2020} for example.

This paper is organized as follows.
In \cref{section:CR-geometry},
we recall basic facts on CR and Sasakian manifolds
and give a proof of \cref{thm:spherical-CR-with-non-trivial-first-Chern-class}.
\cref{section:proof-of-main-theorem} is devoted to a proof of \cref{thm:Chern-classes-of-spherical-CR-structure}.
In \cref{section:related-examples},
we construct some examples related to \cref{thm:Chern-classes-of-spherical-CR-structure,thm:topological-obstruction-to-spherical-CR-structure}.

\medskip

\noindent
\emph{Notation.}
We use Einstein's summation convention
and assume that lowercase Greek indices $\alpha, \beta, \gamma, \dots$ run from $1, \dots, n$.

\medskip

\section{CR geometry}
\label{section:CR-geometry}

\subsection{CR structures}
\label{subsection:CR-structures}

Let $M$ be an orientable smooth $(2n+1)$-dimensional manifold without boundary.
A \emph{CR structure} is a rank $n$ complex subbundle $T^{1,0}M$
of the complexified tangent bundle $TM \otimes \mathbb{C}$ such that
\begin{equation}
	T^{1,0}M \cap T^{0,1}M = 0, \qquad
	[\Gamma(T^{1,0}M), \Gamma(T^{1,0}M)] \subset \Gamma(T^{1,0}M),
\end{equation}
where $T^{0,1}M$ is the complex conjugate of $T^{1,0}M$ in $TM \otimes \bbC$.
A typical example of a CR manifold is a real hypersurface $M$ in an $(n+1)$-dimensional complex manifold $X$;
this $M$ has the canonical CR structure
\begin{equation}
	T^{1,0}M
	\coloneqq T^{1,0}X |_{M} \cap (TM \otimes \bbC).
\end{equation}
In particular,
the unit sphere $S^{2n+1}$ in $\bbC^{n+1}$ has the canonical CR structure $T^{1,0}S^{2n+1}$.
A CR manifold $(M, T^{1,0}M)$ is said to be \emph{spherical}
if it is locally isomorphic to $(S^{2n+1}, T^{1,0} S^{2n+1})$.
Let $K_{M}$ denote the subbundle of $\bigwedge\nolimits^{n + 1} (T^{\ast}M \otimes \bbC)$
defined by the equation $\ovZ \contr \zeta = 0$ for all $Z \in T^{1,0}M$,
called the \emph{canonical bundle} of $M$.
For a real number $w$,
we define the bundle $\calE(w)$ of \emph{$w$-densities} by
\begin{equation}
	\calE(w)
	\coloneqq (K_{M} \otimes \overline{K_{M}})^{-w/(n+2)}.
\end{equation}
(When $K_{M} \setminus \{0\}$ is viewed as a principal $\bbC^{\times}$-bundle,
$\calE(w)$ is the complex line bundle
induced by the representation $\lambda \mapsto \abs{\lambda}^{-2w/(n+2)}$.)

A CR structure $T^{1,0}M$ is said to be \emph{strictly pseudoconvex}
if there exists a nowhere-vanishing real one-form $\theta$ on $M$
such that
$\theta$ annihilates $T^{1,0}M$ and
\begin{equation}
	-\sqrt{- 1} d \theta (Z, \ovZ) > 0, \qquad
	0 \neq Z \in T^{1,0}M.
\end{equation}
We call such a one-form a \emph{contact form}.
The triple $(M, T^{1,0}M, \theta)$ is called a \emph{pseudo-Hermitian manifold}.
Let $T$ denote the \emph{Reeb vector field} with respect to $\theta$; 
that is,
the unique vector field satisfying
\begin{equation}
	\theta(T) = 1, \qquad T \contr d\theta = 0.
\end{equation}
Let $(\tensor{Z}{_{\alpha}})$ be a local frame of $T^{1,0}M$,
and set $\tensor{Z}{_{\ovxa}} = \overline{\tensor{Z}{_{\alpha}}}$.
Then
$(T, \tensor{Z}{_{\alpha}}, \tensor{Z}{_{\ovxa}})$ gives a local frame of $TM \otimes \mathbb{C}$,
called an \emph{admissible frame}.
Its dual frame $(\theta, \tensor{\theta}{^{\alpha}}, \tensor{\theta}{^{\ovxa}})$
is called an \emph{admissible coframe}.
The two-form $d \theta$ is written as
\begin{equation}
	d \theta
	= \sqrt{- 1} \tensor{l}{_{\alpha}_{\ovxb}} \tensor{\theta}{^{\alpha}} \wedge \tensor{\theta}{^{\ovxb}},
\end{equation}
where $(\tensor{l}{_{\alpha}_{\ovxb}})$ is a positive definite Hermitian matrix.
We use $\tensor{l}{_{\alpha}_{\ovxb}}$ and its inverse $\tensor{l}{^{\alpha}^{\ovxb}}$
to raise and lower indices of tensors.

Let $\zeta$ be a (locally defined) nowhere-vanishing section of $K_{M}$.
Let us denote by $\abs{\zeta}^{2w}$ the section $(\zeta \otimes \ovxz)^{w}$ of $\calE(-(n+2)w)$.
A contact form $\theta$ is \emph{volume normalized} by $\zeta$ if
\begin{equation}
	\theta \wedge (d \theta)^{n}
	= (\sqrt{- 1})^{n^{2}}(n !) \theta \wedge (T \contr \zeta) \wedge (T \contr \ovxz).
\end{equation}
Conversely,
for a contact form $\theta$,
there exists a (locally defined) nowhere-vanishing section $\zeta$ of $K_{M}$
such that $\theta$ is volume normalized by $\zeta$.
Such a $\zeta$ is determined uniquely by $\theta$ modulo multiplication by $S^{1}$-valued smooth functions.
Hence $\abs{\zeta}^{2w}$ is independent of the choice of $\zeta$
and defines a nowhere-vanishing global section of $\calE(-(n+2)w)$.
In particular,
$\calE(w)$ is a trivial complex line bundle.

An important example of a strictly pseudoconvex CR manifold
is the boundary of a strictly pseudoconvex domain.
Let $\Omega$ be a relatively compact domain in an $(n+1)$-dimensional complex manifold $X$
with smooth boundary $M = \partial \Omega$.
Then there exists a smooth function $\rho$ on $X$ such that
\begin{equation}
	\Omega = \rho^{-1}((-\infty, 0)), \quad
	M = \rho^{-1}(0), \quad
	d \rho \neq 0 \ \text{on} \ M;
\end{equation}
such a $\rho$ is called a \emph{defining function} of $\Omega$.
A domain $\Omega$ is said to be \emph{strictly pseudoconvex}
if we can take a defining function $\rho$ of $\Omega$ that is strictly plurisubharmonic near $M$.
Then $M$ is a closed strictly pseudoconvex real hypersurface,
and $d^{c} \rho |_{M}$ is a contact form on $M$,
where $d^{c} \coloneqq (\sqrt{-1}/2) (\delb - \del)$.
In particular,
$(S^{2n+1}, T^{1,0}S^{2n+1})$ is a strictly pseudoconvex CR manifold;
this is because $S^{2n+1}$ is the boundary of the unit ball in $\bbC^{n+1}$,
which is a strictly pseudoconvex domain.
This implies that any spherical CR manifold is strictly pseudoconvex.

We call $\Omega$ a \emph{Stein domain}
if $\Omega$ admits a defining function $\rho$
that is strictly plurisubharmonic on a neighborhood of the closure of $\Omega$.
Note that a Stein domain is a Stein manifold;
this is because $-1 / \rho$ defines a strictly plurisubharmonic exhaustion function on $\Omega$.
A co-oriented contact structure $\xi$ on a closed $(2n+1)$-dimensional manifold $M$
is \emph{Stein fillable} if  there exists a Stein domain $\Omega$ of dimension $n+1$
such that $(\bdry \Omega, \Re T^{1,0}\bdry \Omega)$ is contactmorphic to $(M, \xi)$.

\subsection{Tanaka-Webster connection}
\label{subsection:TW-connection}

A contact form $\theta$ induces a canonical connection $\nabla$ on $T M$,
called the \emph{Tanaka-Webster connection} with respect to $\theta$.
It is defined by
\begin{equation}
	\nabla T
	= 0,
	\quad
	\nabla \tensor{Z}{_{\alpha}}
	= \tensor{\omega}{_{\alpha}^{\beta}} \tensor{Z}{_{\beta}},
	\quad
	\nabla \tensor{Z}{_{\ovxa}}
	= \tensor{\omega}{_{\ovxa}^{\ovxb}} \tensor{Z}{_{\ovxb}}
	\quad
	\pqty{ \tensor{\omega}{_{\ovxa}^{\ovxb}}
	\coloneqq \overline{\tensor{\omega}{_{\alpha}^{\beta}}} }
\end{equation}
with the following structure equations:
\begin{gather}
\label{eq:str-eq-of-TW-conn1}
	d \tensor{\theta}{^{\beta}}
	= \tensor{\theta}{^{\alpha}} \wedge \tensor{\omega}{_{\alpha}^{\beta}}
	+ \tensor{A}{^{\beta}_{\ovxa}} \theta \wedge \tensor{\theta}{^{\ovxa}}, \\
\label{eq:str-eq-of-TW-conn2}
	d \tensor{l}{_{\alpha}_{\ovxb}}
	= \tensor{\omega}{_{\alpha}^{\gamma}} \tensor{l}{_{\gamma}_{\ovxb}}
	+ \tensor{l}{_{\alpha}_{\ovxg}} \tensor{\omega}{_{\ovxb}^{\ovxg}}.
\end{gather}
The tensor $\tensor{A}{_{\alpha}_{\beta}} \coloneqq \overline{\tensor{A}{_{\ovxa}_{\ovxb}}}$
is shown to be symmetric and called the \emph{Tanaka-Webster torsion}.

The curvature form
$\tensor{\Omega}{_{\alpha}^{\beta}} \coloneqq d \tensor{\omega}{_{\alpha}^{\beta}}
- \tensor{\omega}{_{\alpha}^{\gamma}} \wedge \tensor{\omega}{_{\gamma}^{\beta}}$
of the Tanaka-Webster connection satisfies
\begin{equation}
\label{eq:curvature-form-of-TW-connection}
	\begin{split}
		\tensor{\Omega}{_{\alpha}^{\beta}}
		&= \tensor{R}{_{\alpha}^{\beta}_{\gamma}_{\ovxs}} \tensor{\theta}{^{\gamma}} \wedge \tensor{\theta}{^{\ovxs}}
			- \tensor{\nabla}{^{\beta}} \tensor{A}{_{\alpha}_{\gamma}} \theta \wedge \tensor{\theta}{^{\gamma}}
			+ \tensor{\nabla}{_{\alpha}} \tensor{A}{^{\beta}_{\ovxg}} \theta \wedge \tensor{\theta}{^{\ovxg}}  \\
		&\quad - \sqrt{-1} \tensor{A}{_{\alpha}_{\gamma}} \tensor{\theta}{^{\gamma}} \wedge \tensor{\theta}{^{\beta}}
			+ \sqrt{-1} \tensor{l}{_{\alpha}_{\ovxg}} \tensor{A}{^{\beta}_{\ovxr}}
			\tensor{\theta}{^{\ovxg}} \wedge \tensor{\theta}{^{\ovxr}}.
	\end{split}
\end{equation}
We call the tensor $\tensor{R}{_{\alpha}^{\beta}_{\gamma}_{\ovxs}}$
the \emph{Tanaka-Webster curvature}.
This tensor has the symmetry 
\begin{equation}
	\tensor{R}{_{\alpha}_{\ovxb}_{\gamma}_{\ovxs}}
	= \tensor{R}{_{\gamma}_{\ovxb}_{\alpha}_{\ovxs}}
	= \tensor{R}{_{\alpha}_{\ovxs}_{\gamma}_{\ovxb}}.
\end{equation}
Contractions of indices give the \emph{Tanaka-Webster Ricci curvature}
$\tensor{\Ric}{_{\gamma}_{\ovxs}} \coloneqq \tensor{R}{_{\alpha}^{\alpha}_{\gamma}_{\ovxs}}$
and the \emph{Tanaka-Webster scalar curvature}
$\Scal \coloneqq \tensor{\Ric}{_{\gamma}^{\gamma}}$.
The \emph{Chern tensor} $\tensor{S}{_{\alpha}_{\ovxb}_{\gamma}_{\ovxs}}$ is defined by
\begin{equation}
	\tensor{S}{_{\alpha}_{\ovxb}_{\gamma}_{\ovxs}}
	\coloneqq \tensor{R}{_{\alpha}_{\ovxb}_{\gamma}_{\ovxs}}
		- \tensor{P}{_{\alpha}_{\ovxb}} \tensor{l}{_{\gamma}_{\ovxs}}
		- \tensor{P}{_{\gamma}_{\ovxb}} \tensor{l}{_{\alpha}_{\ovxs}}
		- \tensor{P}{_{\gamma}_{\ovxs}} \tensor{l}{_{\alpha}_{\ovxb}}
		- \tensor{P}{_{\alpha}_{\ovxs}} \tensor{l}{_{\gamma}_{\ovxb}},
\end{equation}
where
\begin{equation}
	\tensor{P}{_{\alpha}_{\ovxb}}
	\coloneqq \frac{1}{n+2} \pqty{\tensor{\Ric}{_{\alpha}_{\ovxb}} - \frac{\Scal}{2(n+1)} \tensor{l}{_{\alpha}_{\ovxb}}}.
\end{equation}
Note that the trace $P \coloneqq \tensor{P}{_{\alpha}^{\alpha}}$ is equal to $\Scal / 2(n+1)$.
It is known that $(M, T^{1,0}M)$ is a spherical CR manifold
if and only if the Chern tensor vanishes identically when $n \geq 2$~\cite{Chern-Moser1974}.
A tensor $\tensor{V}{_{\alpha}_{\ovxb}_{\gamma}}$ is defined by
\begin{equation}
	\tensor{V}{_{\alpha}_{\ovxb}_{\gamma}}
	\coloneqq \tensor{\nabla}{_{\ovxb}} \tensor{A}{_{\alpha}_{\gamma}}
		+ \sqrt{-1} \tensor{\nabla}{_{\gamma}} \tensor{P}{_{\alpha}_{\ovxb}}
		- \sqrt{-1} \tensor{T}{_{\gamma}} \tensor{l}{_{\alpha}_{\ovxb}}
		- 2 \sqrt{-1} \tensor{T}{_{\alpha}} \tensor{l}{_{\gamma}_{\ovxb}},
\end{equation}
where
\begin{equation}
	\tensor{T}{_{\alpha}}
	\coloneqq\frac{1}{n+2} (\tensor{\nabla}{_{\alpha}} P - \sqrt{-1} \tensor{\nabla}{^{\beta}} \tensor{A}{_{\alpha}_{\beta}}).
\end{equation}
This tensor is a divergence of the Chern tensor:
\begin{equation}
\label{eq:divergence-of-Chern-tensor}
	\tensor{\nabla}{^{\ovxs}} \tensor{S}{_{\alpha}_{\ovxb}_{\gamma}_{\ovxs}}
	= - n \sqrt{-1} \tensor{V}{_{\alpha}_{\ovxb}_{\gamma}};
\end{equation}
see~\cite{Case-Gover2020}*{(2.5)}.

Assume that $n \geq 2$.
A contact form $\theta$ is said to be \emph{pseudo-Einstein} if
\begin{equation}
\label{eq:pseudo-Einstein-condition}
	\tensor{\Ric}{_{\alpha}_{\overline{\beta}}}
	= \frac{1}{n} \Scal \cdot \tensor{l}{_{\alpha}_{\overline{\beta}}}.
\end{equation}
It is known that $c_{1}(T^{1,0}M) = 0$ in $H^{2}(M, \bbR)$
if $(M, T^{1,0}M)$ has a pseudo-Einstein contact form~\cite{Lee1988}*{Proposition D}.

\subsection{Sasakian manifolds}
\label{subsection:Sasakian-manifolds}

Sasakian manifolds play an important role in pseudo-Hermitian geometry.
See~\cite{Boyer-Galicki2008} for a comprehensive introduction to Sasakian manifolds.

A \emph{Sasakian manifold} is a pseudo-Hermitian manifold $(S, T^{1,0}S, \eta)$
with vanishing Tanaka-Webster torsion.
This condition is equivalent to that the Reeb vector field $T$ with respect to $\eta$
preserves the CR structure $T^{1,0}S$.

A typical example of a Sasakian manifold
is the circle bundle associated with a negative Hermitian line bundle.
The construction is as follows.
Let $Y$ be an $n$-dimensional complex manifold
and $(L, h)$ be a Hermitian holomorphic line bundle over $Y$
such that
\begin{equation}
	\omega
	\coloneqq - \sqrt{-1} \Theta_{h} = d d^{c} \log h
\end{equation}
is a \Kahler form on $Y$.
Consider the circle bundle
\begin{equation}
	S
	\coloneqq \Set{ v \in L | h(v, v) = 1 }
\end{equation}
over $Y$,
which is a real hypersurface in the total space of $L$.
The one-form $\eta \coloneqq d^{c} \log h |_{S}$
is a connection one-form of the principal $S^{1}$-bundle $p \colon S \to Y$
and satisfies $d \eta = p^{\ast} \omega$.
Moreover,
the natural CR structure $T^{1,0}S$ coincides with the horizontal lift of $T^{1,0}Y$ with respect to $\eta$.
Since $\omega$ is a \Kahler form,
we have
\begin{equation}
	- \sqrt{-1} d \eta (Z, \ovZ)
	= - \sqrt{-1} \omega (p_{\ast} Z, p_{\ast} \ovZ)
	> 0
\end{equation}
for all non-zero $Z \in T^{1,0}S$.
This implies that $(S, T^{1,0}S)$ is a strictly pseudoconvex CR manifold,
and $\eta$ is a contact form on $S$.
We call this pseudo-Hermitian manifold $(S, T^{1,0}S, \eta)$
the \emph{circle bundle associated with $(Y, L, h)$}.

Consider the Tanaka-Webster connection with respect to $\eta$.
Take a local coordinate $(z^{1}, \dots , z^{n})$ of $Y$.
The \Kahler form $\omega$ is written as
\begin{equation}
	\omega
	= \sqrt{-1} \tensor{g}{_{\alpha}_{\ovxb}} d z^{\alpha} \wedge d \ovz^{\beta},
\end{equation}
where $(\tensor{g}{_{\alpha}_{\ovxb}})$ is a positive definite Hermitian matrix.
Let $Z_{\alpha}$ be the horizontal lift of $\pdvf{}{z^{\alpha}}$.
Then $(T, Z_{\alpha}, Z_{\ovxa} \coloneqq \overline{Z_{\alpha}})$
is an admissible frame on $S$.
The corresponding admissible coframe is given by
$(\eta, \theta^{\alpha} \coloneqq p^{\ast} (d z^{\alpha}), \theta^{\ovxa} \coloneqq p^{\ast} (d \ovz^{\alpha}))$.
Since $d \eta = p^{*} \omega$,
we have
\begin{equation}
	d \eta = \sqrt{-1} (p^{\ast} \tensor{g}{_{\alpha}_{\ovxb}}) \theta^{\alpha} \wedge \theta^{\ovxb},
\end{equation}
which implies $\tensor{l}{_{\alpha}_{\ovxb}} = p^{\ast} \tensor{g}{_{\alpha}_{\ovxb}}$.
The connection form $\tensor{\pi}{_{\alpha}^{\beta}}$ of the \Kahler metric
with respect to the frame $(\pdvf{}{\tensor{z}{^{\alpha}}})$ satisfies
\begin{equation} \label{eq:structure-equation-for-Kahler-metric}
	0 = d (d z^{\beta}) = d \tensor{z}{^{\alpha}} \wedge \tensor{\pi}{_{\alpha}^{\beta}},
	\qquad
	d \tensor{g}{_{\alpha}_{\ovxb}}
	= \tensor{\pi}{_{\alpha}^{\gamma}} \tensor{g}{_{\gamma}_{\ovxb}}
	+ \tensor{g}{_{\alpha}_{\ovxg}} \tensor{\pi}{_{\ovxb}^{\ovxg}}
	\qquad
	\pqty{\tensor{\pi}{_{\ovxa}^{\ovxb}}
	\coloneqq \overline{\tensor{\pi}{_{\alpha}^{\beta}}}}.
\end{equation}
We write as $\tensor{\Pi}{_{\alpha}^{\beta}}$ the curvature form of the \Kahler metric.
Pulling back \cref{eq:structure-equation-for-Kahler-metric} by $p$ gives
\begin{equation}
	d \theta^{\beta} = \tensor{\theta}{^{\alpha}} \wedge (p^{*} \tensor{\pi}{_{\alpha}^{\beta}}),
	\qquad
	d \tensor{l}{_{\alpha}_{\ovxb}}
	= (p^{\ast} \tensor{\pi}{_{\alpha}^{\gamma}}) \tensor{l}{_{\gamma}_{\ovxb}}
	+ \tensor{l}{_{\alpha}_{\ovxg}} (p^{\ast} \tensor{\pi}{_{\ovxb}^{\ovxg}}).
\end{equation}
This yields that
$\tensor{\omega}{_{\alpha}^{\beta}} = p^{\ast} \tensor{\pi}{_{\alpha}^{\beta}}$,
and the Tanaka-Webster torsion vanishes identically;
that is,
$(S, T^{1,0}S, \eta)$ is a Sasakian manifold.
Moreover,
the curvature form $\tensor{\Omega}{_{\alpha}^{\beta}}$ of the Tanaka-Webster connection
is given by $\tensor{\Omega}{_{\alpha}^{\beta}} = p^{\ast} \tensor{\Pi}{_{\alpha}^{\beta}}$.
In particular,
$(S, T^{1, 0} S)$ is a spherical CR manifold
if and only if $\omega$ defines a Bochner-flat \Kahler metric on $Y$ when $n \geq 2$.

Now we construct an example of a closed spherical CR manifold of dimension at least five
having non-trivial first Chern class with real coefficients.

\begin{proof}[Proof of \cref{thm:spherical-CR-with-non-trivial-first-Chern-class}]
	Let $Y_{1}$ be a closed Riemann surface of genus two.
	Take a Hermitian metric $h_{1}$ on $L_{1} \coloneqq T^{1, 0} Y_{1}$
	such that $\omega_{1} \coloneqq - \sqrt{-1} \Theta_{h_{1}}$ defines a \Kahler-Einstein metric on $Y_{1}$
	with Einstein constant $-1$.
	Let $h_{2}$ denote a Hermitian metric on $L_{2} \coloneqq \calO(-2)$ over $Y_{2} \coloneqq \cps^{n-1}$
	such that $\omega_{2} \coloneqq - \sqrt{-1} \Theta_{h_{2}}$
	is a positive constant multiple of the Fubini-Study form.
	Consider the Hermitian holomorphic line bundle
	$(L \coloneqq L_{1} \boxtimes L_{2}, h \coloneqq h_{1} \boxtimes h_{2})$ over $Y \coloneqq Y_{1} \times Y_{2}$.
	Then $\omega \coloneqq - \sqrt{-1} \Theta_{h} = \omega_{1} + \omega_{2}$ defines
	a Bochner-flat \Kahler metric on $Y$~\cite{Tachibana-Liu1970}*{Section 2}.
	Let $(S, T^{1,0}S, \eta)$ be the circle bundle associated with $(Y, L, h)$.
	Since $\omega$ is Bochner-flat,
	$(S, T^{1,0}S)$ is a spherical CR manifold.
	It remains to show $c_{1}(T^{1,0}S) \neq 0$ in $H^{2}(S, \bbR)$.
	The first Chern classes of $L$ and $T^{1,0}Y$ are given by
	\begin{gather}
		c_{1}(L)
		= c_{1}(L_{1}) + c_{1}(L_{2}), \\
		c_{1}(T^{1,0}Y)
		= c_{1}(T^{1,0}Y_{1}) + c_{1}(T^{1,0}Y_{2})
		= c_{1}(L_{1}) - \frac{n}{2} c_{1}(L_{2}).
	\end{gather}
	Consider the Gysin exact sequence
	\begin{equation}
		H^{0}(Y, \bbR) \cong \bbR \xrightarrow{c_{1}(L)}
		H^{2}(Y, \bbR) \xrightarrow{p^{\ast}}
		H^{2}(S, \bbR).
	\end{equation}
	Since $c_{1}(T^{1,0}Y)$ is not proportional to $c_{1}(L)$,
	the cohomology class $p^{\ast} c_{1}(T^{1,0}Y)$ is not equal to zero in $H^{2}(S, \bbR)$.
	From $p^{\ast} T^{1,0}Y \cong T^{1,0}S$,
	it follows that $c_{1}(T^{1,0}S) \neq 0$ in $H^{2}(S, \bbR)$.
	In particular,
	$(S, T^{1,0}S)$ has no pseudo-Einstein contact forms.
\end{proof}

\section{Proof of \cref{thm:Chern-classes-of-spherical-CR-structure}}
\label{section:proof-of-main-theorem}

Let $(M, T^{1,0}M)$ be a strictly pseudoconvex CR manifold of dimension $2 n + 1$.
Denote by $\calT$ the complex vector bundle $\calE(1) \oplus T^{1,0}M \oplus \calE(0)$ of rank $n+2$.
Since both $\calE(1)$ and $\calE(0)$ are trivial line bundles,
we have $c(T^{1,0}M) = c(\calT)$.
Hence it suffices to study $c(\calT)$ for a proof of \cref{thm:Chern-classes-of-spherical-CR-structure}.
In the proof of~\cite{Marugame2021}*{Proposition 5.4},
Marugame has introduced a connection $\nabla^{\calT}$ on $\calT$ via the CR tractor connection~\cite{Gover-Graham2005};
this $\calT$ coincides with $\tensor{\calE}{^{A}}(1,0)$ in~\cite{Marugame2021}.
The curvature form $\Omega^{\calT}$ of $\nabla^{\calT}$ satisfies
\begin{equation}
\label{eq:tractor-curvature}
	\Omega^{\calT}
	=
	\begin{pmatrix}
		0 & 0 & 0 \\
		\ast & \tensor{\Xi}{_{\alpha}^{\beta}} & 0 \\
		\ast & \ast & 0
	\end{pmatrix}
	+ \frac{1}{n + 2} \tensor{\Omega}{_{\gamma}^{\gamma}}
	\begin{pmatrix}
		1 & 0 & 0 \\
		0 & \tensor{\delta}{_{\alpha}^{\beta}} & 0 \\
		0 & 0 & 1
	\end{pmatrix}
	,
\end{equation}
where
\begin{equation}
\label{eq:trace-free-curvature}
	\tensor{\Xi}{_{\alpha}^{\beta}}
	\coloneqq \tensor{S}{_{\alpha}^{\beta}_{\gamma}_{\ovxs}} \tensor{\theta}{^{\gamma}} \wedge \tensor{\theta}{^{\ovxs}}
		- \tensor{V}{_{\alpha}^{\beta}_{\gamma}} \theta \wedge \tensor{\theta}{^{\gamma}}
		+ \tensor{V}{^{\beta}_{\alpha}_{\ovxg}} \theta \wedge \tensor{\theta}{^{\ovxg}}.
\end{equation}

\begin{proof}[Proof of \cref{thm:Chern-classes-of-spherical-CR-structure}]
	Since $(M, T^{1,0}M)$ is spherical,
	the Chern tensor $\tensor{S}{_{\alpha}^{\beta}_{\gamma}_{\ovxs}}$ vanishes identically.
	It follows from \cref{eq:divergence-of-Chern-tensor}
	that so does $\tensor{V}{_{\alpha}^{\beta}_{\gamma}}$ also.
	Thus we have $\tensor{\Xi}{_{\alpha}^{\beta}} = 0$.
	Combining this with \cref{eq:tractor-curvature} yields that
	\begin{equation}
		\det(\id_{\calT} + \frac{\sqrt{- 1}}{2 \pi} \Omega^{\calT})
		= \pqty{ 1 + \frac{1}{n+2} \frac{\sqrt{-1}}{2 \pi} \tensor{\Omega}{_{\gamma}^{\gamma}} }^{n+2}
	\end{equation}
	is a representative of $c(\calT) = c(T^{1,0}M)$.
	Since $(\sqrt{-1} / 2 \pi) \tensor{\Omega}{_{\gamma}^{\gamma}}$ is a representative of $c_{1}(T^{1,0}M)$,
	we have the desired conclusion.
\end{proof}

\section{Related examples}
\label{section:related-examples}

We first construct a spherical CR manifold having non-trivial second Chern class with real coefficients.

\begin{proposition}
\label{prop:spherical-CR-with-non-trivial-second-Chern-class}
	For each integer $n \geq 4$,
	there exists a closed spherical CR manifold $(S, T^{1,0}S)$ of dimension $2n+1$ with
	\begin{equation}
		c_{2}(T^{1,0}S)
		= \frac{n+1}{2(n+2)} c_{1}(T^{1,0}S)^{2}
		\neq 0
	\end{equation}
	in $H^{4}(S, \bbR)$.
\end{proposition}

\begin{proof}
	Let $Y_{1}$ be a fake projective plane;
	that is,
	a closed complex hyperbolic manifold with the same Betti numbers as $\cps^{2}$;
	see~\cite{Mumford1979} for example.
	Take a Hermitian metric $h_{1}$ on $L_{1} \coloneqq K_{Y_{1}}^{-1}$
	such that $\omega_{1} \coloneqq - \sqrt{-1} \Theta_{h_{1}}$ defines a \Kahler metric on $Y_{1}$
	with negative constant holomorphic sectional curvature.
	Let $h_{2}$ denote a Hermitian metric on $L_{2} \coloneqq \calO(-3)$ over $Y_{2} \coloneqq \cps^{n-2}$
	such that $\omega_{2} \coloneqq - \sqrt{-1} \Theta_{h_{2}}$
	is a positive constant multiple of the Fubini-Study form.
	Consider the Hermitian holomorphic line bundle
	$(L \coloneqq L_{1} \boxtimes L_{2}, h \coloneqq h_{1} \boxtimes h_{2})$ over $Y \coloneqq Y_{1} \times Y_{2}$.
	Then $\omega \coloneqq - \sqrt{-1} \Theta_{h} = \omega_{1} + \omega_{2}$ defines
	a Bochner-flat \Kahler metric on $Y$~\cite{Tachibana-Liu1970}*{Section 2}.
	Let $(S, T^{1,0}S, \eta)$ be the circle bundle associated with $(Y, L, h)$.
	Since $\omega$ is Bochner-flat,
	$(S, T^{1,0}S)$ is a spherical CR manifold.
	It suffices to show $c_{1}(T^{1,0}S)^{2} \neq 0$ in $H^{4}(S, \bbR)$
	by \cref{thm:Chern-classes-of-spherical-CR-structure}.
	The first Chern classes of $L$ and $T^{1,0}Y$ are given by
	\begin{gather}
		c_{1}(L)
		= c_{1}(L_{1}) + c_{1}(L_{2}), \\
		c_{1}(T^{1,0}Y)
		= c_{1}(T^{1,0}Y_{1}) + c_{1}(T^{1,0}Y_{2})
		= c_{1}(L_{1}) - \frac{n-1}{3} c_{1}(L_{2}).
	\end{gather}
	Consider the Gysin exact sequence
	\begin{equation}
		H^{2}(Y, \bbR) \xrightarrow{c_{1}(L)}
		H^{4}(Y, \bbR) \xrightarrow{p^{\ast}}
		H^{4}(S, \bbR).
	\end{equation}
	Since $p^{\ast} T^{1,0}Y \cong T^{1,0}S$,
	it is sufficient to prove
	\begin{equation}
	\label{eq:non-triviality-of-second-Chern-class}
		c_{1}(T^{1,0}Y)^{2}
		\notin \Im (H^{2}(Y, \bbR) \xrightarrow{c_{1}(L)} H^{4}(Y, \bbR)).
	\end{equation}
	On the one hand,
	\begin{gather}
		H^{2}(Y, \bbR)
		= \bbR c_{1}(L_{1})
			\oplus \bbR c_{1}(L_{2}), \\
		H^{4}(Y, \bbR)
		= \bbR c_{1}(L_{1})^{2}
			\oplus \bbR c_{1}(L_{1}) c_{1}(L_{2})
			\oplus \bbR c_{1}(L_{2})^{2},
	\end{gather}
	where we use the fact that $n-2 \geq 2$ in the latter equality.
	These imply
	\begin{equation}
		H^{4}(Y, \bbR)
		= \Im (H^{2}(Y, \bbR) \xrightarrow{c_{1}(L)} H^{4}(Y, \bbR))
			\oplus \bbR c_{1}(L_{2})^{2}.
	\end{equation}
	On the other hand,
	\begin{align}
		c_{1}(T^{1, 0} Y)^{2}
		&= \pqty{ c_{1}(L_{1}) - \frac{n-1}{3} c_{1}(L_{2}) }^{2} \\
		&= \pqty{ c_{1}(L) - \frac{n+2}{3} c_{1}(L_{2}) }^{2} \\
		&= c_{1}(L) \pqty{ c_{1}(L) - \frac{2(n+2)}{3}c_{1}(L_{2}) }
			+ \frac{(n+2)^{2}}{9} c_{1}(L_{2})^{2}.
	\end{align}
	Thus we have \cref{eq:non-triviality-of-second-Chern-class}.
\end{proof}

We next show that \cref{thm:Chern-classes-of-spherical-CR-structure} does not hold
for Chern classes with integer coefficients in general.
Note that the following example is the same as that in the proof of~\cite{Takeuchi2020}*{Proposition 4.1}.

\begin{proof}[Proof of \cref{prop:counterexample-to-integer-coefficients}]
	Fix a positive integer $d$.
	Let $h$ be a Hermitian metric on the holomorphic line bundle $\calO(-d)$ over $\cps^{n}$
	such that $\omega \coloneqq - \sqrt{-1} \Theta_{h}$
	is a positive constant multiple of the Fubini-Study form.
	Denote by $(S, T^{1,0}S, \eta)$ the circle bundle associated with $(\cps^{n}, \calO(-d), h)$.
	Since the Fubini-Study metric is Bochner-flat,
	$(S, T^{1,0}S)$ is a spherical CR manifold.
	Let $p \colon S \to \cps^{n}$ be the canonical projection.
	To simplify notation,
	we write $\tau$ for $c_{1}(\calO(1))$,
	which is a generator of $H^{2}(\cps^{n}, \bbZ) \cong \bbZ$.
	On the one hand,
	consider the Gysin exact sequence
	\begin{equation}
		H^{2}(\cps^{n}, \bbZ) \xrightarrow{-d \cdot \tau}
		H^{4}(\cps^{n}, \bbZ) \xrightarrow{p^{\ast}}
		H^{4}(S, \bbZ) \xrightarrow{}
		H^{3}(\cps^{n}, \bbZ) = 0.
	\end{equation}
	Since $H^{4}(\cps^{n}, \bbZ)$ is freely generated by $\tau^{2}$,
	the above exact sequence implies that $H^{4}(S, \bbZ) \cong \bbZ / d \bbZ$
	and $p^{\ast} \tau^{2}$ is a generator of $H^{4}(S, \bbZ)$.
	On the other hand,
	since $p^{\ast} T^{1,0}\cps^{n}$ is isomorphic to $T^{1,0}S$ as a complex vector bundle,
	\begin{equation}
		c_{1}(T^{1,0}S)
		= (n+1) p^{\ast} \tau,
		\qquad
		c_{2}(T^{1, 0} S)
		= \frac{n(n+1)}{2} p^{\ast} \tau^{2}.
	\end{equation}
	Hence
	\begin{align}
		2(n+2) c_{2}(T^{1,0}S) - (n+1) c_{1}(T^{1,0}S)^{2}
		&= n(n+1)(n+2) p^{\ast} \tau^{2} - (n+1)^{3} p^{\ast} \tau^{2} \\
		&= -(n+1) p^{\ast} \tau^{2}.
	\end{align}
	Therefore if we choose $d$ as a prime integer greater than $n+1$,
	then we have
	\begin{equation}
		 2(n+2) c_{2}(T^{1,0}S) 
		 \neq (n+1) c_{1}(T^{1,0}S)^{2}
	\end{equation}
	in $H^{4}(S, \bbZ)$.
\end{proof}

Now we consider the problem whether a given contact manifold admits a spherical CR structure.
We can explicitly construct a Stein fillable contact manifold without spherical CR structures.

\begin{proposition}
\label{prop:Reinhardt-boundary-without-spherical}
	Let $M$ be the boundary of the Stein domain
	\begin{equation}
		\Omega
		\coloneqq \Set{ z = (z^{0}, \dots, z^{n}) \in \bbC^{n+1} | \sum_{i = 0}^{n} (\log \abs{z^{i}})^{2} < 1 }
	\end{equation}
	in $\bbC^{n+1}$.
	Then $(M, \Re T^{1,0}M)$ is Stein fillable but admits no spherical CR structures.
\end{proposition}

\begin{proof}
	On the one hand,
	$M$ is diffeomorphic to $S^{n} \times T^{n+1}$.
	Hence its fundamental group $\pi_{1}(M)$ is given by
	\begin{equation}
		\pi_{1}(M)
		\cong
		\begin{cases}
			\bbZ^{3} & n = 1, \\
			\bbZ^{n+1} & n \geq 2.
		\end{cases}
	\end{equation}
	In particular,
	$\pi_{1}(M)$ is Abelian.
	On the other hand,
	there is a classification of spherical CR manifolds
	with nilpotent fundamental group by Goldman~\cite{Goldman1983}*{Theorem 4.1};
	$M$ must be finitely covered by $S^{2n+1}$, $S^{1} \times S^{2n}$, or $\bbH^{2n+1} / \Gamma$.
	Here $\bbH^{2n+1}$ is the Heisenberg group of dimension $2n+1$
	and $\Gamma \subset \bbH^{2n+1}$ is a discrete cocompact subgroup.
	Considering the universal cover of $M$ yields that
	the first and second cases do not occur.
	Hence $M$ must be finitely covered by $\bbH^{2n+1} / \Gamma$.
	However,
	the fundamental group of $\bbH^{2n+1} / \Gamma$ is $\Gamma$,
	which is non-Abelian;
	see~\cite{Folland2004}*{Section 2} for example.
	This contradicts the fact that $\pi_{1}(M)$ is Abelian.
	Therefore $M$ has no spherical CR structures.
\end{proof}

However,
the total Chern class $c(T^{1,0}M)$ is trivial for the above example.
Our last example is a Stein fillable contact manifold violating the equality \cref{eq:second-Chern-contact-with-spherical}.
Note that the following example is the same as that in the proof of~\cite{Takeuchi2020}*{Proposition 4.2}.

\begin{proof}[Proof of \cref{prop:Stein-fillable-with-no-spherical-CR}]
	Let $\Omega_{0}$ be a Stein domain in a two-dimensional complex manifold $X_{0}$
	such that its boundary $M_{0} \coloneqq \partial \Omega_{0}$
	satisfies $c_{1}(T^{1,0}M_{0}) \neq 0$ in $H^{2}(M_{0}, \mathbb{R})$;
	see~\cite{Etnyre-Ozbagci2008}*{Theorem 6.2} for an example of such $\Omega_{0}$.
	Take a defining function $\rho$ of $\Omega_{0}$
	that is strictly plurisubharmonic near the closure of $\Omega_{0}$.
	Without loss of generality,
	we may assume that $\rho$ is an exhaustion function on $X_{0}$.
	Then,
	for sufficiently small $\epsilon$,
	there exists a diffeomorphism
	$\chi \colon (- \epsilon, \epsilon) \times M_{0}
	\to \rho^{- 1}((- \epsilon, \epsilon))$
	satisfying $\chi(0, p) = p$ and $\rho \circ \chi (t, p) = t$.
	The function $\psi_{0} \coloneqq -1 / \rho$
	gives a strictly plurisubharmonic exhaustion function on $\Omega_{0}$.

	We first show the statement for the case of odd $n$;
	write $n = 2m-1$, $m \geq 2$.
	Consider the domain
	\begin{equation}
		\Omega
		\coloneqq \Set{ (p_{1}, \dots , p_{m}) \in (\Omega_{0})^{m}
		| \psi_{0}(p_{1}) + \dots + \psi_{0}(p_{m}) < 2m / \epsilon }.
	\end{equation}
	The function $\psi(p_{1}, \dots , p_{m}) \coloneqq \psi_{0}(p_{1}) + \dots + \psi_{0}(p_{m})$
	is a strictly plurisubharmonic exhaustion function on $(\Omega_{0})^{m}$,
	and $d \psi \neq 0$ on $M = \partial \Omega$.
	Hence $\Omega$ is a Stein domain
	in $(\Omega_{0})^{m} \subset (X_{0})^{m}$.
	Consider the map $\iota \colon (M_{0})^{m} \to M (\subset (X_{0})^{m})$
	defined by
	\begin{equation}
		\iota(p_{1}, \dots , p_{m})
		\coloneqq (\chi(- \epsilon / 2, p_{1}), \dots , \chi(- \epsilon / 2, p_{m})).
	\end{equation}
	Since this map is homotopic to the natural embedding $(M_{0})^{m} \hookrightarrow (X_{0})^{m}$,
	\begin{equation}
		\iota^{\ast} c(T^{1,0}M)
		= c(\iota^{\ast} T^{1,0}(X_{0})^{m})
		= c(T^{1,0}(X_{0})^{m} |_{(M_{0})^{m}})
		= c((T^{1,0}M_{0})^{m}).
	\end{equation}
	For $1 \leq i \leq m$,
	let $p_{i} \colon (M_{0})^{m} \to M_{0}$ be the $i$-th projection,
	and set
	\begin{equation}
		\tau_{i}
		\coloneqq p_{i}^{\ast} c_{1}(T^{1,0}M_{0}) \in H^{2}((M_{0})^{m}, \bbR).
	\end{equation}
	Then we have
	\begin{equation}
		\iota^{\ast} c(T^{1,0}M)
		= c \pqty{ \bigoplus_{i = 1}^{m} p_{i}^{\ast} T^{1,0}M_{0}}
		= \prod_{i = 1}^{m} p_{i}^{\ast} c(T^{1,0}M_{0})
		= \prod_{i = 1}^{m} (1 + \tau_{i}).
	\end{equation}
	In particular,
	\begin{equation}
		\iota^{\ast} c_{1}(T^{1,0}M)
		= \sum_{i = 1}^{m} \tau_{i},
		\qquad
		\iota^{\ast} c_{2}(T^{1,0}M)
		= \sum_{i < j} \tau_{i} \tau_{j}
	\end{equation}
	From the fact that $c_{1}(T^{1,0}M_{0}) \neq 0$ in $H^{2}(M_{0}, \bbR)$,
	it follows that $\iota^{*} c_{2}(T^{1,0}M) \neq 0$ in $H^{4}((M_{0})^{m}, \bbR)$.
	Moreover,
	by using $\tau_{i}^{2} = 0$,
	we have
	\begin{equation}
		\iota^{\ast} c_{2}(T^{1,0}M)
		= \frac{1}{2} \iota^{\ast} c_{1}(T^{1,0}M)^{2}
		\neq \frac{n+1}{2(n+2)} \iota^{\ast} c_{1}(T^{1,0}M)^{2}.
	\end{equation}
	This implies that
	the contact structure $\xi = \Re T^{1,0}M$ satisfies
	\begin{equation}
		c_{2}(\xi)
		= c_{2}(T^{1,0}M)
		\neq \frac{n+1}{2(n+2)} c_{1}(T^{1,0}M)^{2}
		= \frac{n + 1}{2(n+2)} c_{1}(\xi)^{2}
	\end{equation}
	in $H^{4}(M, \bbR)$.

	We next treat the case of even $n$;
	write $n = 2m$, $m \geq 2$.
	Consider the domain
	\begin{equation}
		\Omega
		\coloneqq \Set{(p_{1}, \dots , p_{m}, z) \in (\Omega_{0})^{m} \times \mathbb{C}
		| \psi_{0}(p_{1}) + \dots + \psi_{0}(p_{m}) + |z|^{2} < 2m / \epsilon}.
	\end{equation}
	This $\Omega$ is a Stein domain in $(\Omega_{0})^{m} \times \mathbb{C} \subset (X_{0})^{m} \times \mathbb{C}$.
	Consider the map $\iota \colon (M_{0})^{m} \to M = \partial \Omega$ given by
	\begin{equation}
		\iota(p_{1}, \dots , p_{m})
		\coloneqq (\chi(- \epsilon / 2, p_{1}), \dots , \chi(- \epsilon / 2, p_{m}), 0).
	\end{equation}
	Then we obtain
	\begin{equation}
		\iota^{\ast} c(T^{1,0}M)
		= c((T^{1,0}M_{0})^{m}).
	\end{equation}
	Similar to the case of odd $n$,
	we have
	\begin{equation}
		c_{2}(\xi)
		\neq \frac{n+1}{2(n+2)} c_{1}(\xi)^{2}
	\end{equation}
	in $H^{4}(M, \bbR)$.
	This proves the statement.
\end{proof}

\section*{Acknowledgements}
The author is grateful to Taiji Marugame and Yoshihiko Matsumoto for helpful comments.

\bibliography{my-reference,my-reference-preprint}

\end{document}